\documentclass[letterpaper, 10 pt, conference]{ieeeconf}  

\IEEEoverridecommandlockouts                              
\usepackage{graphicx} 

\usepackage{amsmath}
\usepackage{amssymb}

\title{\LARGE \bf
Dynamic Watermarking for General LTI Systems}

\author{Pedro Hespanhol, Matthew Porter, Ram Vasudevan, and Anil Aswani
\thanks{This work was supported by the UC Berkeley Center for Long-Term Cybersecurity, and by Ford Motor Company.}
\thanks{Pedro Hespanhol and Anil Aswani are with the Department of Industrial Engineering and Operations Research, University of California, Berkeley, CA 94720, USA 
        {\tt\small pedrohespanhol@berkeley.edu, aaswani@berkeley.edu}}%
				\thanks{Matthew Porter and Ram Vasudevan are with the Department of Mechanical Engineering, University of Michigan, Ann Arbor, MI 48109, USA 
        {\tt\small matthepo@umich.edu, ramv@umich.edu}}%
}

\def\doubleunderline#1{\underline{\underline{#1}}}
\DeclareMathOperator{\alim}{as-lim}
\DeclareMathOperator{\rank}{rank}
\DeclareMathOperator{\diag}{diag}
\DeclareMathOperator{\trace}{trace}

\newtheorem{lemma}{Lemma}
\newtheorem{theorem}{Theorem}

\newtheorem{proposition}{Proposition}

\newtheorem{corollary}{Corollary}

\begin{document}

\maketitle
\thispagestyle{empty}
\pagestyle{empty}

\begin{abstract}
Detecting attacks in control systems is an important aspect of designing secure and resilient control systems.  Recently, a dynamic watermarking approach was proposed for detecting malicious sensor attacks for SISO LTI systems with partial state observations and MIMO LTI systems with a full rank input matrix and full state observations; however, these previous approaches cannot be applied to general LTI systems that are MIMO and have partial state observations. This paper designs a dynamic watermarking approach for detecting malicious sensor attacks for general LTI systems, and we provide a new set of asymptotic and statistical tests.  We prove these tests can detect attacks that follow a specified attack model (more general than replay attacks), and we also show that these tests simplify to existing tests when the system is SISO or has full rank input matrix and full state observations.  The benefit of our approach is demonstrated with a simulation analysis of detecting sensor attacks in autonomous vehicles.  Our approach can distinguish between sensor attacks and wind disturbance (through an internal model principle framework), whereas improperly designed tests cannot distinguish between sensor attacks and wind disturbance.
\end{abstract}

\section{INTRODUCTION}

Secure and resilient control requires the development of mechanisms to allow safe operation in the face of malicious attacks or external interferences. This is particularly challenging for cyber-physical systems (CPS) that feature interconnection between physical sensors and actuators with the communication and computation capabilities of routers, servers, etc.  Such concerns are motivated by real-world instances of attacks on CPS, including: the Maroochy-Shire incident \cite{abrams2008malicious}, the Stuxnet worm \cite{langner2011stuxnet}, and other incidents \cite{cardenas2008research}.

For control systems, two possible modes of attacks are either an attacker inserting faulty measurements into the output sensor signal or an attacker inserting malicious values into the actuator input for the control system. Cybersecurity techniques \cite{parno2006secure,kumar2006managing,wang2013cyber,kim2012cyber} are an important component of designing resilient CPS.  However, CPS frequently has a decentralized structure; and so approaches that detect attacks by decoupling different sensing and actuating components of the system are particularly useful for ensuring safe operation.

This paper designs a dynamic watermarking approach for detecting malicious sensor attacks for general LTI systems, and has two main contributions: First, we generalize the watermarking approach developed in \cite{satchidanandan2016dynamic} for SISO LTI systems with partial state observations and MIMO LTI systems with a full rank input matrix and full state observations under an arbitrary attack, and our generalization applies to general LTI systems under a specific attack model that is more general than replay attacks \cite{weerakkody2014detecting}. Second, we show that modeling is important for designing watermarking techniques: For instance, dynamic watermarking was used to detect sensor attacks in an intelligent transportation system \cite{ko2016theory}; however, here we show that persistent disturbances such as those from wind can invalidate watermarking approaches, and we propose an approach based on the internal model principle to compensate for persistent disturbances.  This second contribution motivates our generalization of dynamic watermarking to general MIMO LTI systems with partial observations, since internal model states are never directly observed.

\subsection{Watermarking for CPS}

Defense and security for CPS is classified into either detection and identification \cite{cardenas2008research,cardenas2008secure,pasqualetti2013attack}, and a number of ``passsive'' techniques have been proposed.  State estimation algorithms \cite{fawzi2014secure,fawzi2011secure} have been suggested in order to handle attacks on the physical plants within a CPS. Another related approach \cite{bai2015security} provides a metric to characterize the resilience of a system facing stealthy attacks on the actuators.


More recently, ``active defense'' based on watermarking has been developed for detecting sensor attacks \cite{satchidanandan2016dynamic,ko2016theory,mo2009secure,mo2010false,mo2014detecting,weerakkody2014detecting,mo2015physical}.  The idea is that honest (i.e., not compromised by an attacker) actuators superimpose a random signal onto the control input to ensure security in face of sensor attacks.  One set of approaches \cite{mo2009secure,mo2010false,mo2014detecting,weerakkody2014detecting,mo2015physical} develops statistical hypothesis tests that detect attacks with a certain error rate, while dynamic watermarking approaches \cite{satchidanandan2016dynamic,ko2016theory} develop a test to ensure that only attacks which add a zero-average-power signal to the sensor measurements can remain undetected.  The first set of techniques applies to general LTI systems under specific attack models, but cannot ensure the zero-average-power property for attacks; while the second set of techniques applies to specific LTI systems under general attack models.  Our first contribution in this paper is to partially bridge the gap between these two techniques by developing a method that applies to general LTI systems under specific attack models and that ensures the zero-average-power property for attacks.

\subsection{Security for Intelligent Transportation Systems}

The design of intelligent transportation systems (ITS) is receiving increased attention \cite{ko2016theory,gonzalez2010perpetual,aswani2011,zhang2012hierarchical,vasudevan2012safe,mohan2016convex,como2016convexity}, and one significant area for further study is the design of methods to ensure the safe and resilient operation of ITS.  One recent work \cite{ko2016theory} considered the use of dynamic watermarking to detect sensor attacks in a network of autonomous vehicles coordinated by a supervisory controller; the watermarking approach was successfully able to detect attacks.  However, large-scale deployments of ITS must be resilient in the face of persistent disturbances from environmental and human factors.  Wind is an example of such a persistent disturbance.  A second contribution of this work is from the perspective of modeling: We show that persistent disturbances such as those from wind can invalidate watermarking approaches, and we propose an internal model principle-based approach to handle persistent disturbances.  This motivates our generalization of dynamic watermarking to general MIMO LTI systems with partial observations, since internal model states are not directly observed.

\subsection{Outline}
Section \ref{sec:ltisam} reviews the general LTI system model (i.e., MIMO systems with partial observations) and specifies our attack model, and Sect. \ref{sec:idnt} provides intuition on why existing dynamic watermarking approaches cannot by used on a general LTI system.   We construct a detection consistent dynamic watermarking approach for general LTI systems under our attack model in Sect. \ref{sec:dct}, and our term \emph{detection consistent} test is used to refer to a test that ensures the zero-average-power property (described above) for attacks.  Next, Sect. \ref{sec:svt} describes how our asymptotic tests can be converted into statistical tests, and Sect. \ref{sec:ret} shows how our tests are special cases of those in \cite{satchidanandan2016dynamic} for the SISO case or the MIMO case with full rank input matrix and full state observations.  We conclude with Sect. \ref{sec:sav}, which conducts simulations of an autonomous vehicle: Our tests are able to distinguish between sensor attacks and wind disturbances when including wind disturbance in the system dynamics using the internal model principal, while improperly designed tests cannot distinguish between attacks and wind.

\section{LTI System and Attack Model}

\label{sec:ltisam}

Let $[r] = \{1,\ldots,r\}$, and consider a MIMO LTI system $x_{n+1} = Ax_n + Bu_n + w_n$ with partial observations $y_n = Cx_n + z_n + v_n$, where $x \in\mathbb{R}^p$, $u\in\mathbb{R}^q$, and $y,z,v\in\mathbb{R}^m$.  The $v_n$ should be interpreted as an additive measurement disturbance added by an attacker, while $w_n$ represents zero mean i.i.d. process noise with a jointly Gaussian distribution and covariance $\Sigma_W$, and $z_n$ represents zero mean i.i.d. measurement noise with a jointly Gaussian distribution and covariance $\Sigma_Z$.  We further assume the process noise is independent of the measurement noise, that is $w_n$ for $n\geq0$ is independent of $z_n$ for $n\geq0$.

If $(A,B)$ is stabilizable and $(A,C)$ is detectable, then a stabilizing output-feedback controller can be designed when $v_n\equiv 0$ using an observer and the separation principle.  Let $K$ be a constant state-feedback gain matrix such that $A+BK$ is Schur stable, and let $L$ be a constant observer gain matrix such that $A+LC$ is Schur stable.  The idea of dynamic watermarking in this context will be to superimpose a private (and random) excitation signal $e_n$ known in value to the controller but unknown in value to the attacker.  As a result, we will apply the control input $u_n = K\hat{x}_n + e_n$, where $\hat{x}_n$ is the observer-estimated state and $e_n$ are i.i.d. Gaussian with zero mean and constant variance $\Sigma_E$ fixed by the controller.

Let $\tilde{x}^\mathsf{T} = \begin{bmatrix} x^\mathsf{T} & \hat{x}^\mathsf{T}\end{bmatrix}$, and define $\underline{B}^\mathsf{T} = \begin{bmatrix} B^\mathsf{T} & B^\mathsf{T}\end{bmatrix}$, $\underline{C} = \begin{bmatrix} C & 0\end{bmatrix}$, $\underline{D}^\mathsf{T} = \begin{bmatrix} \mathbb{I} & 0\end{bmatrix}$, $\underline{L}^\mathsf{T} = \begin{bmatrix} 0 & -L^\mathsf{T}\end{bmatrix}$, and 
\begin{equation}
\underline{A} = \begin{bmatrix} A & BK \\ -LC & A+BK+LC\end{bmatrix}
\end{equation}
Then the closed-loop system with private excitation is given by $\tilde{x}_{n+1} = \underline{A}\tilde{x}_n + \underline{B}e_n + \underline{D}w_n + \underline{L}(z_n+v_n)$.  If we define the observation error $\delta = \hat{x}-x$, then with the change of variables $\check{x}^\textsf{T} = \begin{bmatrix} x^\textsf{T} & \delta^\textsf{T}\end{bmatrix}$ we have the dynamics $\check{x}_{n+1} = \doubleunderline{A}\check{x}_n + \doubleunderline{B}e_n + \doubleunderline{D}w_n + \doubleunderline{L}(z_n+v_n)$, where $\doubleunderline{B}^\textsf{T} = \begin{bmatrix}B^\textsf{T} & 0\end{bmatrix}$, $\doubleunderline{D}^\textsf{T} = \begin{bmatrix} \mathbb{I} & -\mathbb{I}\end{bmatrix}$, $\doubleunderline{L} = \underline{L}$, and 
\begin{equation}
\doubleunderline{A} = \begin{bmatrix} A+BK & BK \\ 0 & A+LC\end{bmatrix}.
\end{equation}
Recall that $\doubleunderline{A}$ is Schur stable whenever $A+BK$ and $A+LC$ are both Schur stable.

Since the controller is fixed, we can suppose the attacker chooses $v_n = \alpha(Cx_n + z_n) + C\xi_n + \zeta_n$ for some fixed $\alpha \in \mathbb{R}$, where $\xi_{n+1} = (A+BK)\xi_n + \omega_n$, $\zeta_n$ are i.i.d. Gaussian with zero mean and constant variance $\Sigma_S$ fixed by the attacker, and $\omega_n$ are i.i.d. Gaussian with zero mean and constant variance $\Sigma_O$ fixed by the attacker.  The idea underlying this attack model is that the attacker allows some fraction of the true output $Cx_n + z_n$ to be measured by the controller, and at the same time also incorporates the measurement of a false state $\xi_n$ that evolves according the dynamics that would be expected under the controller.

\section{Intuition for Designing a New Test}

\label{sec:idnt}

To better understand how to design a new test, it is instructive to apply existing dynamic watermarking schemes and the associated tests \cite{satchidanandan2016dynamic} to particular LTI systems.  Such an exercise provides intuition that we use to design new tests.  Our main example is an LTI system with
\begin{equation}
\begin{aligned}
A = \begin{bmatrix}1 & 1 \\ 0 & 1\end{bmatrix},\ B=\begin{bmatrix}0\\1\end{bmatrix},\text{ and } C=\begin{bmatrix}1 & 0\end{bmatrix}.
\end{aligned}
\end{equation}
Suppose the attacker chooses $v_n = -(Cx_n + z_n) + C\xi_n + \zeta_n$ with $\Sigma_S = \Sigma_Z$ and $\Sigma_O=\Sigma_W$, meaning the output measurement $y_n = C\xi_n + \zeta_n$ has no component from the actual system.  This is a SISO (i.e., $m=q=1$) system with partial state measurement, and the tests in \cite{satchidanandan2016dynamic} pass for this example, even though the sensor has been compromised by an attacker.  The problem in this example is that the test
\begin{equation}
\textstyle\alim_{N} \frac{1}{N}\sum_{n=0}^{N-1} L(C\hat{x}_n^{\vphantom{\textsf{T}}} - y_n^{\vphantom{\textsf{T}}}) e_{n-1}^\textsf{T} = 0
\end{equation}
from \cite{satchidanandan2016dynamic} correlates the innovations process $L(C\hat{x}_n- y_n)$ with the private excitation only one step back in time $e_{n-1}$; however, it takes two time steps for the control input to enter into the output in this example.  And so when designing a new test for general LTI systems, we need to take into consideration that there is generally some delay between when some private excitation is applied to when it is observed.

\section{Detection Consistent Test}

\label{sec:dct}

Now let $\Sigma_X$ be the positive semidefinite matrix that solves the following
\begin{equation}
\Sigma_X = \doubleunderline{A}\Sigma_X\doubleunderline{A}^\textsf{T} + \doubleunderline{B}\Sigma_E\doubleunderline{B}^\textsf{T} + \doubleunderline{D}\Sigma_W\doubleunderline{D}^\textsf{T} + \doubleunderline{L}\Sigma_Z\doubleunderline{L}^\textsf{T}.
\end{equation}
Note that $\Sigma_X = \alim_N \frac{1}{N}\sum_{n=0}^{N-1} \check{x}_n^{\vphantom{\textsf{T}}}\check{x}_n^\textsf{T}$.  Similarly let $\Sigma_\Delta$ be the positive semidefinite matrix that solves the following
\begin{equation}
\Sigma_\Delta = (A+LC)\Sigma_\Delta(A+LC)^\textsf{T} + \Sigma_W + L\Sigma_ZL^\textsf{T}.
\end{equation}
Note $\Sigma_\Delta = \alim_N \frac{1}{N}\sum_{n=0}^{N-1} \delta_n^{\vphantom{\textsf{T}}}\delta_n^\textsf{T}$ and $\Sigma_\Delta = \doubleunderline{M}\Sigma_X\doubleunderline{M}^\textsf{T}$, where $\doubleunderline{M} = \begin{bmatrix} 0 & \mathbb{I}\end{bmatrix}$.  Recall that $\Sigma_X$ and $\Sigma_\Delta$ exist because the above are Lyapunov equations with matrices $\doubleunderline{A}, (A+LC)$ that are Schur stable.

\begin{lemma}
\label{lemma:ar}
We have that
\begin{equation}
\underline{A}^r\underline{B} = \begin{bmatrix}(A+BK)^rB\\(A+BK)^rB\end{bmatrix}
\end{equation}
for all $r\geq 0$
\end{lemma}

\begin{proof}
The result holds for $r = 0$ since $\underline{A}^0 = \mathbb{I}$ and $(A+BK)^0 = \mathbb{I}$.  Now suppose the result holds for $r$: We prove that it holds for $r + 1$.  In particular, note that
\begin{equation}
\underline{A}^{r+1}\underline{B} = \underline{A}\begin{bmatrix}(A+BK)^rB\\(A+BK)^rB\end{bmatrix} = \begin{bmatrix}(A+BK)^{r+1}B\\(A+BK)^{r+1}B\end{bmatrix},
\end{equation}
where the first equality holds by the inductive hypothesis, and the second equality follows by calculation of the matrix multiplication.  Hence the result follows by induction.
\end{proof}

\begin{proposition}
\label{prop:aalph}
Let $\underline{A}(\alpha) = \underline{A}+\alpha\underline{H}$ with
\begin{equation}
\label{eqn:math}
\underline{H} = \begin{bmatrix} 0 & 0 \\ -LC & 0\end{bmatrix},
\end{equation}
and define $k' = \min\{k \geq 0\ |\ C(A+BK)^kB \neq 0\}$.  Then we have that $\underline{A}(\alpha)^k\underline{B} = \underline{A}^k\underline{B}$ for $0 \leq k \leq k'$.
\end{proposition}

\begin{proof}
If $k' = 0$, then the result holds trivially.  So assume $k' \geq 1$.  We have that $\underline{A}(\alpha)^0\underline{B} = \underline{A}^0\underline{B}=\underline{B}$ since $\underline{A}(\alpha)^0 = \underline{A}^0 = \mathbb{I}$.  Now suppose $\underline{A}(\alpha)^k\underline{B} = \underline{A}^k\underline{B}$ for $0\leq k \leq k'-1$.  But using Lemma \ref{lemma:ar} implies that
\begin{multline}
\underline{A}(\alpha)^{k+1}\underline{B} = \underline{A}^{k+1}\underline{B} + \alpha \underline{H}\begin{bmatrix}(A+BK)^kB\\(A+BK)^kB\end{bmatrix} = \\
\underline{A}^{k+1}\underline{B} + \alpha\begin{bmatrix}0\\-LC(A+BK)^kB\end{bmatrix} = \underline{A}^{k+1}\underline{B},
\end{multline}
where we have used that $LC(A+BK)^kB = 0$ since $k<k'$.  And so the result follows by induction.
\end{proof}

Now let $k' = \min\{k\geq 0\ |\ C(A+BK)^kB \neq 0\}$, and consider the following tests
\begin{multline}
\label{eqn:acttest1}
\textstyle\alim_{N} \frac{1}{N}\sum_{n=0}^{N-1} (C\hat{x}_n-y_n)^{\vphantom{\textsf{T}}}(C\hat{x}_n-y_n)^\textsf{T} = \\
C\Sigma_\Delta C^\textsf{T}+\Sigma_Z
\end{multline}
\begin{equation}
\label{eqn:acttest2}
\textstyle\alim_{N} \frac{1}{N}\sum_{n=0}^{N-1} (C\hat{x}_n^{\vphantom{\textsf{T}}} - y_n^{\vphantom{\textsf{T}}}) e_{n-k'-1}^\textsf{T} = 0.
\end{equation}

\begin{theorem}
\label{thm:detcons}
Suppose $(A,B)$ is stabilizable, $(A,C)$ is detectable, $\Sigma_E$ is full rank, and $k' = \min\{k\geq 0\ |\ C(A+BK)^kB \neq 0\}$ exists.  If the test (\ref{eqn:acttest1})--(\ref{eqn:acttest2}) holds, then 
\begin{equation}
\textstyle\alim_N \frac{1}{N}\sum_{n=0}^{N-1} v_n^\textsf{T}v_n^{\vphantom{\textsf{T}}} = 0,
\end{equation}
meaning that $v_n$ asymptotically has zero power.
\end{theorem}

\begin{proof}
Observe that the dynamics for $\tilde{x}$ are given by $\tilde{x}_{n+1} = \underline{A}(\alpha)\cdot\tilde{x}_n + \underline{B}e_n + \underline{D}w_n + \underline{L}((1+\alpha)z_n+C\xi_n+\zeta_n)$,
where $\underline{A}(\alpha) = \underline{A}+\alpha\underline{H}$ with $\underline{H}$ given in (\ref{eqn:math}).  Next note that a basic calculation gives
\begin{multline}
\textstyle\tilde{x}_{n} = \underline{A}(\alpha)^k\tilde{x}_{n-k} + \sum_{k'=0}^{k-1}\underline{A}(\alpha)^{k-k'-1}\big(\underline{B}e_{n+k'-k} + \\
\underline{D}w_{n+k'-k} + (1+\alpha)\cdot\underline{L}z_{n+k'-k} + \\
\textstyle\underline{L}C\xi_{n+k'-k} + \underline{L}\zeta_{n+k'-k}\big).
\end{multline}
If we define $\doubleunderline{C} = \begin{bmatrix} -C & C\end{bmatrix}$, then $C\hat{x}_n - y_n = \doubleunderline{C}\tilde{x}_n - \alpha\cdot\underline{C}\tilde{x}_n - (1+\alpha)\cdot z_n -C\xi_n-\zeta_n$, and so for $k\in[p]$ we have
\begin{multline}
\label{eqn:exptest1}
\textstyle\frac{1}{N}\sum_{n=0}^{N-1} \mathbb{E}\big((C\hat{x}_n^{\vphantom{\textsf{T}}} - y_n^{\vphantom{\textsf{T}}})e_{n-k}^\textsf{T}\big) =\\ (\doubleunderline{C}-\alpha\cdot\underline{C})\cdot\underline{A}(\alpha)^{k-1}\underline{B}\Sigma_E.
\end{multline}
Note that $k' \leq p-1$ by the Cayley-Hamilton theorem.  So combining Proposition \ref{prop:aalph} with (\ref{eqn:exptest1}) implies 
\begin{equation}
\label{eqn:quan1}
\begin{aligned}
\textstyle\frac{1}{N}\sum_{n=0}^{N-1} \mathbb{E}\big((C\hat{x}_n^{\vphantom{\textsf{T}}} - y_n^{\vphantom{\textsf{T}}}) e_{n-k'-1}^\textsf{T}\big) &= (\doubleunderline{C}-\alpha\cdot\underline{C})\cdot\underline{A}^{k'}\underline{B}\Sigma_E \\
&=-\alpha\cdot\underline{C}\cdot\underline{A}^{k'}\underline{B}\Sigma_E
\end{aligned}
\end{equation}
where the second equality holds by by Lemma \ref{lemma:ar} and the definition of $\doubleunderline{C}$.  Because the test (\ref{eqn:acttest2}) holds, the quantity (\ref{eqn:quan1}) should equal $0$.  But since $\Sigma_E$ is full rank by assumption, Sylvester's rank inequality implies $\underline{C}\cdot\underline{A}^{k'}\underline{B}\Sigma_E \neq 0$ since
\begin{equation}
\underline{C}\cdot\underline{A}^{k'}\underline{B} = \begin{bmatrix}0\\C(A+BK)^{k'}B\end{bmatrix} \neq 0,
\end{equation}
where the first equality holds by Lemma \ref{lemma:ar} and the definition of $\underline{C}$.  Thus we must have $\alpha = 0$.

Next consider the expression 
\begin{multline}
\textstyle\frac{1}{N}\sum_{n=0}^{N-1}(C\hat{x}_n-y_n)^{\vphantom{\textsf{T}}}(C\hat{x}_n-y_n)^\textsf{T} = \\
\textstyle\frac{1}{N}\sum_{n=0}^{N-1}(C\hat{x}_n-(1+\alpha)\cdot(Cx_n+z_n)-C\xi_n - \zeta_n)^{\vphantom{\textsf{T}}}\times\\
\textstyle(C\hat{x}_n-(1+\alpha)\cdot(Cx_n+z_n)-C\xi_n - \zeta_n)^\textsf{T}.
\end{multline}
We showed above that $\alpha = 0$, and so the expectation of the above expression is
\begin{multline}
\label{eqn:longexp}
\textstyle C\Sigma_\Delta C^\textsf{T} + \Sigma_Z + \Sigma_S + \frac{1}{N}\sum_{n=0}^{N-1}\mathbb{E}\big(C\xi_n^{\vphantom{\textsf{T}}}\xi_n^\textsf{T} C^\textsf{T}\big) + \\
\textstyle\frac{1}{N}\sum_{n=0}^{N-1}C(A+BK)^{N-1}x_0(C(A+BK)^{N-1}\xi_0)^\textsf{T}.
\end{multline}
Since $(A+BK)$ is Schur stable, the associated property of exponential stability implies
\begin{equation}
\label{eqn:limcross}
\textstyle\lim_N \frac{1}{N}\sum_{n=0}^{N-1}C(A+BK)^{N-1}x_0(C(A+BK)^{N-1}\xi_0)^\textsf{T} = 0
\end{equation}
by combining the Cauchy-Schwartz inequality with the exponential stability.  However from the test (\ref{eqn:acttest1}), the expectation must equal $C\Sigma_\Delta C^\textsf{T} + \Sigma_Z$ in the limit.  Since all the terms in the above expectation (\ref{eqn:longexp}) are positive semidefinite or have zero limit, this implies 
\begin{equation}
\label{eqn:exptestv}
\textstyle \Sigma_S + \alim_N \frac{1}{N}\sum_{n=0}^{N-1}\mathbb{E}\big(C\xi_n^{\vphantom{\textsf{T}}}\xi_n^\textsf{T} C^\textsf{T}\big) = 0.
\end{equation}
Finally, consider the expression
\begin{multline}
\label{eqn:covv}
\textstyle\frac{1}{N}\sum_{n=0}^{N-1} v_n^{\vphantom{\textsf{T}}}v_n^\textsf{T} = \frac{1}{N}\sum_{n=0}^{N-1} \big((\alpha(Cx_n + z_n) + C\xi_n + \zeta_n\big)\times\\
\big(\alpha(Cx_n + z_n) + C\xi_n + \zeta_n)^\textsf{T}\big).
\end{multline}
Since $\alpha =0$, the expectation of the above expression is
\begin{multline}
\label{eqn:expcovv}
\textstyle\Sigma_S + \frac{1}{N}\sum_{n=0}^{N-1}\mathbb{E}\big(C\xi_n^{\vphantom{\textsf{T}}}\xi_n^\textsf{T} C^\textsf{T}\big) + \\
\textstyle\frac{1}{N}\sum_{n=0}^{N-1}C(A+BK)^{N-1}x_0(C(A+BK)^{N-1}\xi_0)^\textsf{T}.
\end{multline}
Combining (\ref{eqn:limcross})--(\ref{eqn:expcovv}) implies $\alim_N \frac{1}{N}\sum_{n=0}^{N-1} v_n^{\vphantom{\textsf{T}}}v_n^\textsf{T} = 0$.  However, $v_n^\textsf{T}v_n^{\vphantom{\textsf{T}}}$ equals the sum of the diagonal entries of $v_n^{\vphantom{\textsf{T}}}v_n^\textsf{T}$.  Thus we have $\alim_N\frac{1}{N}\sum_{n=0}^{N-1} v_n^\textsf{T}v_n^{\vphantom{\textsf{T}}} = 0$.
\end{proof}

Existence of $k' = \min\{k\geq 0\ |\ C(A+BK)^kB\neq 0\}$ is easy to verify because Cayley-Hamilton implies $k' \leq p-1$ or it does not exist, but we also give sufficient conditions.

\begin{corollary}
Suppose $(A,B)$ is controllable, $(A,C)$ is observable, and $\Sigma_E$ is full rank.  If the test (\ref{eqn:acttest1})--(\ref{eqn:acttest2})  holds, then 
\begin{equation}
\textstyle\alim_N \frac{1}{N}\sum_{n=0}^{N-1} v_n^\textsf{T}v_n^{\vphantom{\textsf{T}}} = 0,
\end{equation}
meaning that $v_n$ asymptotically has zero power.
\end{corollary}

\begin{proof}
We claim that, under the conditions stated, $k' = \min\{k\geq 0\ |\ C(A+BK)^kB \neq 0\} \leq p-1$ exists.  Indeed, since $(A,B)$ is controllable we have that: $(A+BK,B)$ is controllable, and the controllability matrix
\begin{equation}
\mathfrak{C} = \begin{bmatrix} B & (A+BK)B & \ldots & (A+BK)^{p-1}B\end{bmatrix}
\end{equation}
has $\rank(\mathfrak{C}) = p$.  And so by Sylvester's rank inequality, we have $\rank(C\mathfrak{C}) \geq \rank(C) + \rank(\mathfrak{C}) - p = \rank(C)$. But $(A,C)$ is observable, and so the observability matrix
\begin{equation}
\mathfrak{O} = \begin{bmatrix} C \\ CA \\ \vdots \\ CA^{p-1}\end{bmatrix} = \diag(C,\ldots,C)\begin{bmatrix} \mathbb{I} \\ A \\ \vdots \\ A^{p-1}\end{bmatrix}
\end{equation}
has $\rank(\mathfrak{O}) = p$.  Again applying Sylvester's rank inequality implies $p\rank(C) \geq \rank(\mathfrak{O}) = p$, or equivalently that $\rank(C) \geq 1$.  Combining this with the earlier inequality gives $\rank(C\mathfrak{C}) \geq 1$, and so $C\mathfrak{C} \neq 0$.  This means $k' \leq p-1$ exists since $C\mathfrak{C}$ is a block matrix consisting of the blocks $C(A+BK)^kB$.  Thus the result follows by Theorem \ref{thm:detcons}.
\end{proof}

\section{Statistical Version of Test}

\label{sec:svt}

For the purpose of implementation, we can also construct a statistical version of our test (\ref{eqn:acttest1})--(\ref{eqn:acttest2}).  Our approach is similar to \cite{satchidanandan2016dynamic} in that we construct a hypothesis test by thresholding the negative log-likelihood.  Before defining the test, we make the following useful observation:
\begin{proposition}
Let $\psi^\textsf{T}_n = \begin{bmatrix} (C\hat{x}_n^{\vphantom{\textsf{T}}} - y_n^{\vphantom{\textsf{T}}})^\textsf{T} & e_{n-k'-1}^\textsf{T}\end{bmatrix}$.  The test (\ref{eqn:acttest1})--(\ref{eqn:acttest2}) holds if and only if the following test holds:
\begin{equation}
\label{eqn:acttestjoint}
\textstyle\alim_{N} \frac{1}{N}\sum_{n=0}^{N-1} \psi_n^{\vphantom{\textsf{T}}}\psi_n^\textsf{T} = \begin{bmatrix}C\Sigma_\Delta C^\textsf{T}+\Sigma_Z & 0\\ 0 & \Sigma_E\end{bmatrix}.
\end{equation}
Moreover, if the test (\ref{eqn:acttest1})--(\ref{eqn:acttest2}) holds or equivalently the test (\ref{eqn:acttestjoint}) holds, then we have that $\alim_n \mathbb{E}(\psi_n) = 0$.
\end{proposition}

\begin{proof}
The equivalence between (\ref{eqn:acttest1})--(\ref{eqn:acttest2}) and (\ref{eqn:acttestjoint}) follows from the definition of $\psi_n$ and of the tests.  Next suppose either (equivalent) test holds: Using the dynamics on $\check{x}$ we have $\mathbb{E}(\check{x}_{n+1}) = \doubleunderline{A}\mathbb{E}(\check{x}_n) + \doubleunderline{L}\mathbb{E}(v_n)$.  But we have $v_n = C(A+BK)^n\xi_0 + C\sum_{k=0}^{n-1}(A+BK)^{n-k-1}\omega_k + \zeta_n$ since $\alpha = 0$ as shown in the proof of Theorem \ref{thm:detcons}, and so $\mathbb{E}(v_n) = C(A+BK)^n\xi_0$.  Since $(A+BK)$ is Schur stable, we have $\lim_n\mathbb{E}(v_n) = 0$ and hence $\lim_n\mathbb{E}(\check{x}_n) = \doubleunderline{A}\lim_n\mathbb{E}(\check{x}_n)$.  This means that $\lim_n\mathbb{E}(\check{x}_n) = 0$ since $\mathbb{I}-\doubleunderline{A}$ is full rank (which can be seen by recalling that $\doubleunderline{A}$ is Schur stable, so cannot have any eigenvalue of exactly one, and thus $\det(s\mathbb{I}-A) \neq 0$ for $s = 1$).  Since $C\hat{x}_n-y_n = \begin{bmatrix}0 & C\end{bmatrix}\check{x}_n$, we have that $\mathbb{E}(C\hat{x}_n-y_n) = 0$.  This implies $\mathbb{E}(\psi_n) = 0$ since $\mathbb{E}(e_{n-k'-1})=0$ by construction.
\end{proof}

This result implies that asymptotically the summation $S_n = \frac{1}{\ell}\sum_{n+1}^{n+\ell}\psi_n^{\vphantom{\textsf{T}}}\psi_n^\textsf{T}$ with $\ell \geq m+q$ has a Wishart distribution with $\ell$ degrees of freedom and a scale matrix that matches (\ref{eqn:acttestjoint}), and we use this observation to define a statistical test.  In particular, we check if the negative log-likelihood 
\begin{multline}
\label{eqn:nlltest}
\mathcal{L}(S_n) = (m+q+1-\ell)\cdot\log\det S_n + \\
\trace\left(\begin{bmatrix}(C\Sigma_\Delta C^\textsf{T}+\Sigma_Z)^{-1} &0\\ 0 & \Sigma_E^{-1}\end{bmatrix}\times S_n\right)
\end{multline}
corresponding to this Wishart distribution and the summation $S_n$ is large by conducting the hypothesis test
\begin{equation}
\begin{cases}
\text{reject, } &\text{if } \mathcal{L}(S_n) > \tau(\alpha)\\
\text{accept, } &\text{if } \mathcal{L}(S_n) \leq \tau(\alpha)
\end{cases}
\end{equation}
where $\tau(\alpha)$ is a threshold that controls the false error rate $\alpha$.  A rejection corresponds to the detection of an attack, while an acceptance corresponds to the lack of detection of an attack.  This notation emphasizes the fact that achieving a specified false error rate $\alpha$ (a false error in our context corresponds to detecting an attack when there is no attack occurring) requires changing the threshold $\tau(\alpha)$.

\section{Relationship to Existing Tests}

\label{sec:ret}

It is interesting to compare our test (\ref{eqn:acttest1})--(\ref{eqn:acttest2}) to those designed in \cite{satchidanandan2016dynamic}.  More specifically, \cite{satchidanandan2016dynamic} designed a related sequence of tests adapted to different (and less complex) assumptions about the model dynamics.  We will show that our test is closely related to (and generalizes) these previous tests developed under assumptions of less complex dynamics.

The simplest test in \cite{satchidanandan2016dynamic} was designed for systems with direct state measurement (i.e., $C = \mathbb{I}$), no measurement error (i.e, $z_n\equiv 0$), and full rank input matrix (i.e, $\rank(B) = p$).  The SISO (i.e., $m=p=q=1$) and MIMO cases were considered separately in \cite{satchidanandan2016dynamic}, though the SISO case is a special case of the MIMO case.  If we choose $L = -A$, then we have that: $y_n = x_n + v_n$, $x_{n+1} = Ax_n + BK\hat{x}_n + Be_n + w_n$, $\hat{x}_{n+1} = Ay_n + BK\hat{x}_n + Be_n$, $\Sigma_\Delta = \Sigma_W$, and $k'=1$ since $\rank(CB) = p$.  So our test (\ref{eqn:acttest1})--(\ref{eqn:acttest2}) simplifies to
\begin{multline}
\textstyle\alim_{N} \frac{1}{N}\sum_{n=0}^{N-1} (y_{n+1} - Ay_{n} -BK\hat{x}_{n} - Be_{n})\times\\ (y_{n+1} - Ay_{n} -BK\hat{x}_{n} - Be_{n})^\textsf{T}= 
\Sigma_W
\end{multline}
\begin{multline}
\label{eqn:simptest2}
\textstyle\alim_{N} \frac{1}{N}\sum_{n=0}^{N-1} (y_{n+1} - Ay_{n} -BK\hat{x}_{n} - Be_{n})\times\\e_{n}^\textsf{T} = 0.
\end{multline}
This exactly matches the test designed in \cite{satchidanandan2016dynamic} for LTI systems with the above described properties.

A more complex test in \cite{satchidanandan2016dynamic} was designed for SISO (i.e., $m=q=1$) systems with partial state measurement.  In our notation, the tests in \cite{satchidanandan2016dynamic} for this case simplify to
\begin{multline}
\label{eqn:sd1}
\textstyle\alim_{N} \frac{1}{N}\sum_{n=0}^{N-1} L(C\hat{x}_n-y_n)^{\vphantom{\textsf{T}}}(C\hat{x}_n-y_n)^\textsf{T}L^\textsf{T} = \\
L\big(C\Sigma_\Delta C^\textsf{T}+\Sigma_Z\big)L^\textsf{T}
\end{multline}
\begin{equation}
\label{eqn:sd2}
\textstyle\alim_{N} \frac{1}{N}\sum_{n=0}^{N-1} L(C\hat{x}_n^{\vphantom{\textsf{T}}} - y_n^{\vphantom{\textsf{T}}}) e_{n-1}^\textsf{T} = 0.
\end{equation}
But $k'=1$ since $B$ is a nonzero vector and $\rank(CB) = 1$ in this case.  So the test (\ref{eqn:sd1})--(\ref{eqn:sd2}) from \cite{satchidanandan2016dynamic} essentially matches our test (\ref{eqn:acttest1})--(\ref{eqn:acttest2}), but with the difference that the test in \cite{satchidanandan2016dynamic} considers quantities with $L(C\hat{x}_n-y_n)$, while our test directly considers quantities with $C\hat{x}_n-y_n$; this is a negligible difference since $C\hat{x}_n-y_n$ is a scalar in this SISO case.

\begin{figure}
\label{fig:one}
\centering
\includegraphics[trim={0.2in 0in 0.3in 0in},clip,scale=0.95]{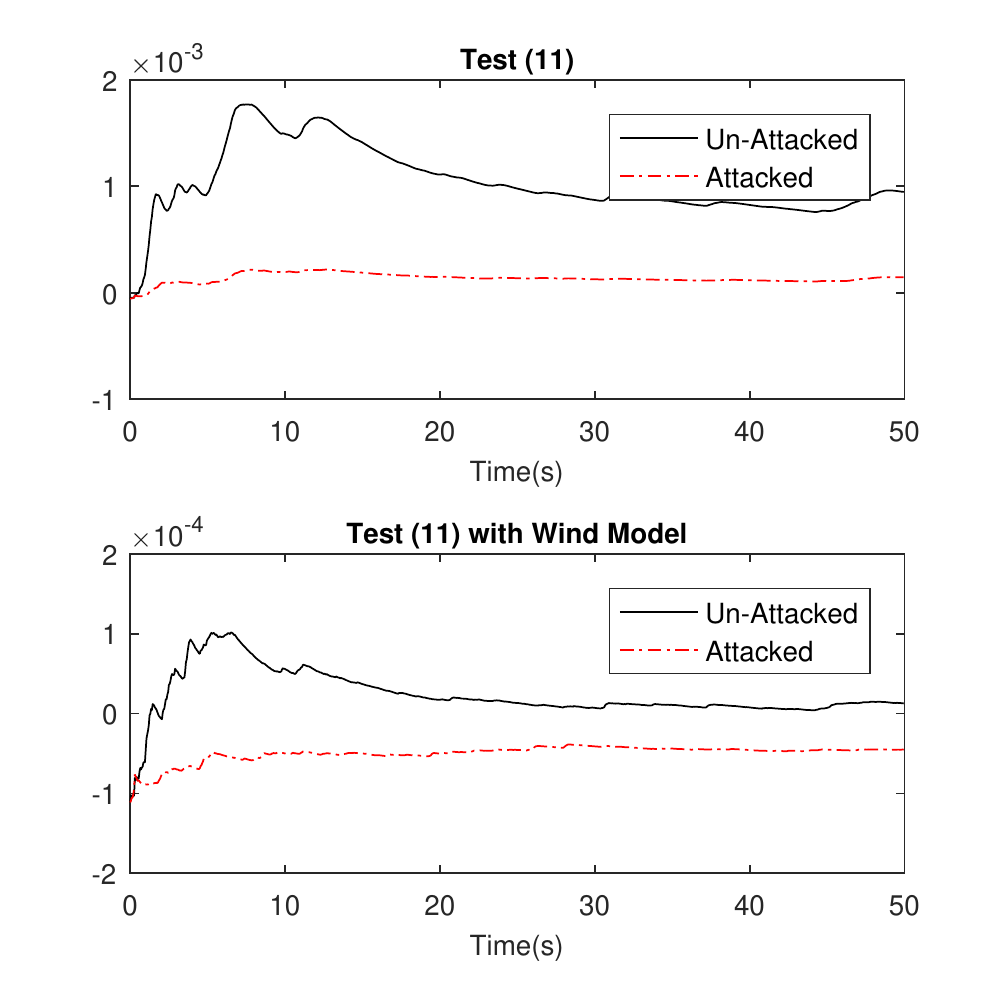}
\caption{Deviation of (\ref{eqn:acttest1}) in Simulation of Autonomous Vehicle}
\end{figure}
\begin{figure}
\label{fig:two}
\includegraphics[trim={0.2in 0in 0.3in 0in},clip,scale=0.95]{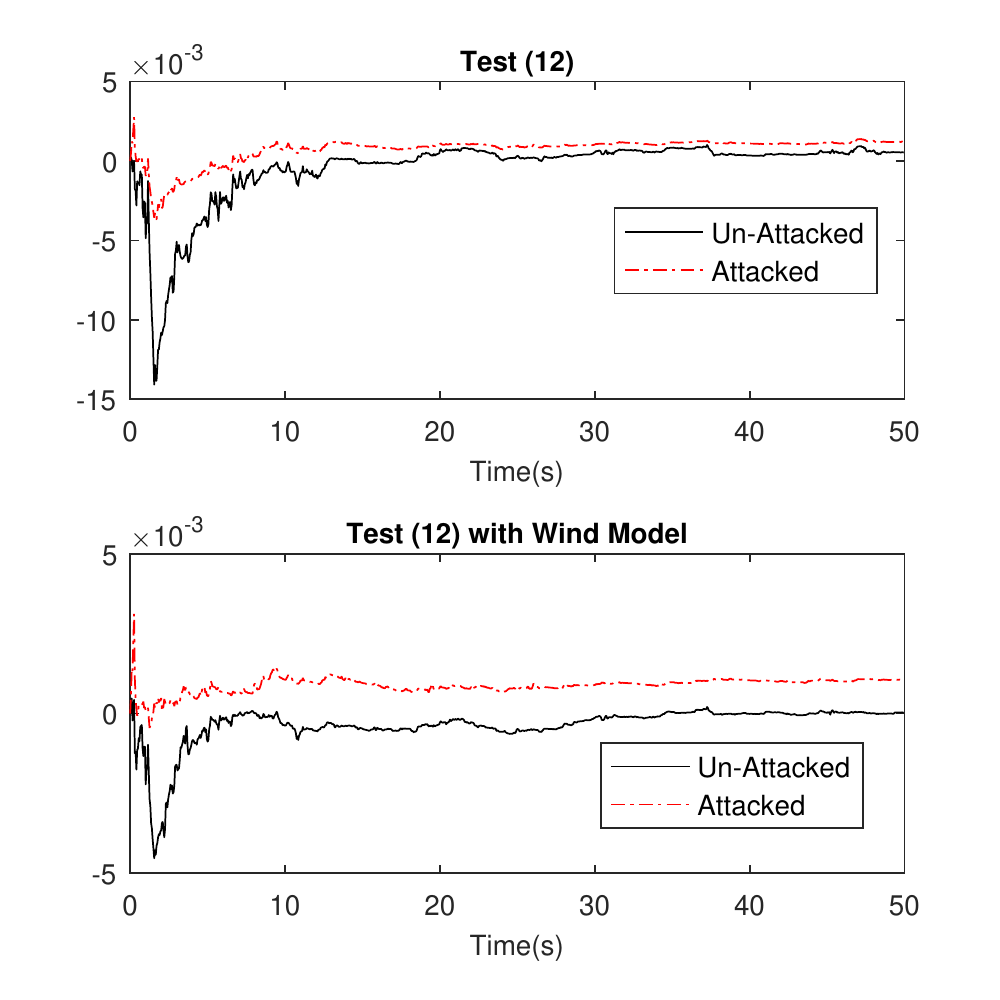}
\caption{Deviation of (\ref{eqn:acttest2}) in Simulation of Autonomous Vehicle}
\end{figure}

\section{Simulations: Autonomous Vehicle}

\label{sec:sav}

A standard model \cite{turri2013linear} for error kinematics of lane keeping and speed control has $x^\textsf{T} = \begin{bmatrix} \psi\ y \  s\ \gamma\ v\end{bmatrix}$ and $u^\textsf{T} = \begin{bmatrix}r\ a\end{bmatrix}$, where $\psi$ is heading error, $y$ is lateral error, $s$ is trajectory distance, $\gamma$ is vehicle angle, $v$ is vehicle velocity, $r$ is steering, and $a$ is acceleration. Linearizing about a straight trajectory and constant velocity $v_0 = 10$, and then performing exact discretization with sampling period $t_s = 0.05$ yields
\begin{equation}
A = \begin{bmatrix} 1 & 0 & 0 & \frac{1}{10} & 0\\
\frac{1}{2} & 1 & 0 & \frac{1}{40} & 0\\
0 & 0 & 1 & 0 & \frac{1}{2}\\
0 & 0 & 0 & 1 & 0\\
0 & 0 & 0 & 0 & 1 \end{bmatrix}\qquad B = \begin{bmatrix} \frac{1}{400} & 0\\
\frac{1}{2400} & 0\\
0 & \frac{1}{800}\\
\frac{1}{20} & 0\\
0 & \frac{1}{20} \end{bmatrix}
\end{equation}
with $C = \begin{bmatrix} I & 0\end{bmatrix} \in\mathbb{R}^{3\times5}$.  We used process and measurement noise with $\Sigma_W = 10^{-8}$ and $\Sigma_Z = 10^{-5}$, respectively.  Our simulations used the wind model: $d_{n+1} = 0.9d_n + \chi_n$, where $\chi_n$ are i.i.d. zero mean Gaussians with $\sigma^2_\chi = 2\times 10^{-6}$, and the wind state $d$ entered additively into the $y$ dynamics.

We applied our tests using a dynamic watermark with variance $\Sigma_E = \frac{1}{2}\mathbb{I}$, where $K$ and $L$ were chosen to stabilize the closed-loop system without an attack.  We conducted four simulations: Un-attacked and attacked simulations were conducted with a test computed without wind in the system model, and un-attacked and attacked simulations were conducted with a test computed with wind in the system model.  In both attack simulations, we chose an attacker with $\alpha = -0.6$, $\xi_0 = 0$, $\Sigma_O =10^{-8}$, and $\Sigma_S = 10^{-8}$.  Fig. \ref{fig:one} shows $\|\frac{1}{N}\sum_{n=0}^{N-1} (C\hat{x}_n-y_n)^{\vphantom{\textsf{T}}}(C\hat{x}_n-y_n)^\textsf{T} -C\Sigma_\Delta C^\textsf{T}-\Sigma_Z\|$, and Fig. \ref{fig:two} shows $\|\frac{1}{N}\sum_{n=0}^{N-1} (C\hat{x}_n^{\vphantom{\textsf{T}}} - y_n^{\vphantom{\textsf{T}}}) e_{n-k'-1}^\textsf{T}\|$.  If the test is detection consistent, then these values go to zero.  The plots show dynamic watermarking cannot detect the presence or absence of an attack when wind affects the system dynamics but is not included in the test, while our test sdetect the presence or absence of an attack when a model of wind is included in the test.  Fig. \ref{fig:three} shows the results of applying our statistical test (\ref{eqn:nlltest}), and the same behavior is seen.

\begin{figure*}
\label{fig:three}
\includegraphics[trim={0.5in 0in 0.7in 0in},clip]{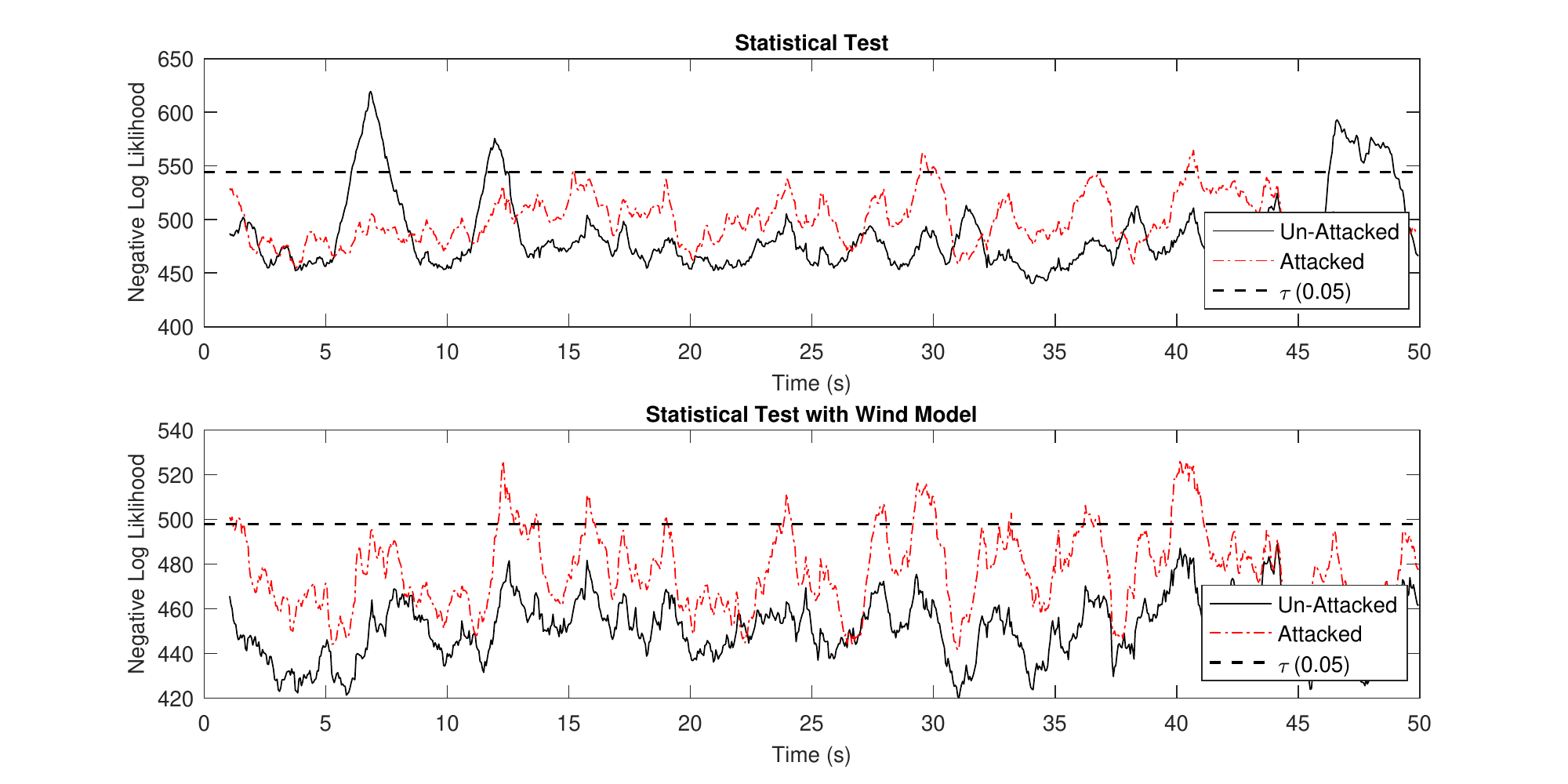}
\caption{Value of (\ref{eqn:nlltest}) for Simulation of Autonomous Vehicle, with a Negative Log-Likelihood Threshold for $\alpha=0.05$ False Detection Error Rate}
\end{figure*}

\section{Conclusion}

This paper constructed a dynamic watermarking approach for detecting malicious sensor attacks for general LTI systems, and the two main contributions were: to extend dynamic watermarking to general LTI systems under a specific attack model that is more general than replay attacks, and to show that modeling is important for designing watermarking techniques by demonstrating how persistent disturbances can negatively affect the accuracy of dynamic watermarking.  Our approach to resolve this issue was to incorporate a model of the persistent disturbance via the internal model principle.  Future work includes generalizing the attack models that can be detected by our approach.  An additional direction for future work is to study the problem of robust controller design in the regime of when an attack is detected.

\bibliographystyle{IEEEtran}
\bibliography{IEEEabrv,secure}

\end{document}